\documentclass[11pt]{article}%graphicx,color

\RequirePackage[l2tabu, orthodox]{nag}

\usepackage[affil-it]{authblk}

\usepackage{stmaryrd}
\usepackage{amscd}
\usepackage[all]{xy}
\usepackage{pb-diagram}
\usepackage{ascmac}
\usepackage{hyperref}
\usepackage{amsrefs}
\usepackage[a4paper]{geometry}
\usepackage{docmute}
\usepackage{enumerate}
\usepackage{comment}
\usepackage{amsmath,amssymb}
\usepackage{amsthm}
\usepackage{mathrsfs}%\mathscr
\usepackage{bm}%\bm
\usepackage{mathtools}%\coloneqq
\usepackage{colonequals}%\colonequals
\usepackage{graphicx}
\usepackage{subcaption}
\usepackage{tikz}
\usetikzlibrary{patterns}
\usepackage{tikz-cd}
\usepackage{cancel}
\usepackage{centernot}%\centernot
\usepackage{xparse}%\NewDocumentCommand
\usepackage{array,booktabs,float}
\usepackage{textcomp}
\usepackage{gensymb}%\degree
\usepackage{bbm}%\mathbbm
\usepackage{cleveref}

\DeclareFontFamily{U}{matha}{\hyphenchar\font45}
\DeclareFontShape{U}{matha}{m}{n}{
<5> <6> <7> <8> <9> <10> gen * matha
<10.95> matha10 <12> <14.4> <17.28> <20.74> <24.88> matha12
}{}

\DeclareSymbolFont{matha}{U}{matha}{m}{n}
\DeclareMathSymbol{\divides}{3}{matha}{"17}
\DeclareMathSymbol{\notdivides}{3}{matha}{"1F}

\theoremstyle{plain}
\newtheorem{Def}{Definition}[section]
\newtheorem*{Def*}{Definition}%def
\newtheorem{thm}[Def]{Theorem}
\newtheorem{cor}[Def]{Corollary}
\newtheorem*{cor*}{Corollary}

\newtheorem*{thm*}{Theorem}

\newtheorem*{prop*}{Proposition}
\newtheorem{lem}[Def]{Lemma}
\newtheorem*{lem*}{Lemma}

\newtheorem*{rem*}{Remark}
\newtheorem{ex}[Def]{Example}
\newtheorem*{ex*}{Example}

\newtheorem{que}{Question}

\newcommand{\Z}{\mathbb{Z}}
\newcommand{\Q}{\mathbb{Q}}

\newcommand{\C}{\mathbb{C}}
\newcommand{\F}{\mathbb{F}}

\newcommand{\abs}[1]{\left\lvert#1\right\rvert}

\newcommand{\id}{\mathrm{id}}

\newcommand{\Bil}{\mathrm{Bil}}
\newcommand{\Quad}{\mathrm{Quad}}
\newcommand{\QVQZ}{\mathrm{Quad} (V, \mathbb{Q}/\mathbb{Z})}
\newcommand{\Qz}{\mathrm{Quad}_0 (V, \mathbb{Q}/\mathbb{Z})}
\newcommand{\QnG}{\mathrm{Quad} (\theta (V^n), \mathbb{Q}/\mathbb{Z})_G}
\newcommand{\QG}{\mathrm{Quad}_0 (\theta (V^n), \mathbb{Q}/\mathbb{Z})_G}

\newcommand{\GL}{\mathrm{GL}}
\newcommand{\hwe}{\mathrm{hwe}}
\newcommand{\cwe}{\mathrm{cwe}}
\newcommand{\fwe}{\mathrm{fwe}}
\newcommand{\CW}{\mathcal{C}}
\newcommand{\wt}{\mathrm{wt}}
\newcommand{\QZ}{\Q/\Z}
\newcommand{\ceq}{\coloneqq}

\newcommand{\rv}[2]{{#1} = ({#1}_1, \dots, {#1}_{#2})}

\title{SELF-DUAL CODES WITH GROUP ACTIONS AND INVARIANTS}
\author{FUTO TAKABAYASHI \thanks{\texttt{wismuth83@keio.jp}}}
\affil{Department of Mathematics, Faculty of Science and Technology, Keio University, Japan}
\date{}

\begin{document}

\maketitle

\begin{abstract}

In this paper, we define dual codes over arbitrary finite rings with respect to arbitrary bilinear forms and provide a generalization of Hayden's theorem (Bridges, Hall, and Hayden, 1981). 
Building on this foundation, we introduce the concept of $G$-dual codes for codes invariant under a permutation group $G$, referred to as $G$-codes. 
We then present several generalizations of Atsumi's MacWilliams identity (Atsumi, 1995; Chakraborty and Miezaki, 2023) for $G$-codes over finite rings with respect to general bilinear forms. 
Furthermore, we establish a $G$-analogue of the MacWilliams identity for $G$-full weight enumerators and introduce the notions of $G$-quadratic maps and $G$-representations for twisted modules, twisted rings, quadratic pairs, and form rings. 
By defining transformation groups for $G$-full weight enumerators, we extend the theory of Clifford--Weil groups (Nebe, Rains, and Sloane, 2004, 2006). 
Finally, we provide generalizations of Gleason-type theorems for these weight enumerators, demonstrating that the $G$-full weight enumerators of $G$-self-dual and $G$-isotropic codes are invariant under the Clifford--Weil groups and span the invariant subspaces of these groups. 

\end{abstract}

\vspace{1em}
\noindent \textbf{Keywords:} Invariant theory, Coding theory, Weight enumerators, $G$-codes, Clifford--Weil groups.

%\vspace{1em}

\section*{Statements and Declarations}

\subsection*{Competing Interests}

The authors have no relevant financial or non-financial interests to disclose.

\subsection*{Funding}

This work was supported by The Keio University Doctorate Student Grant-in-Aid Program from Ushioda Memorial Fund 2025.

%%%%%%%%%%%%%%%%%%%%%%%%%%%%%%%%%%%%%%%%%%%%%%%%%%%%%%%%%%%%%%%%%%%%%%%%%%%%

\section{Introdution}

In coding theory, there are numerous studies of codes equipped with group actions. 
In particular, the properties of codes equipped with permutations are important, as permutations act as automorphisms of codes. 

Yoshida \cite{Y93} showed a generalized MacWilliams identity for codes over $\F_q$ that are invariant under a permutation group $G$, or $G$-codes over $\F_q$. 
Atsumi \cite{A95} gave an analogue of the identity for $G$-codes over $\F_q$ using Hayden's theorems \cite[Theorem 4.4, Corollary 1]{BHH81}. 
Later, Chakraborty and Miezaki presented several generalizations of Atsumi's MacWilliams identity for $G$-codes over finite Frobenius rings. 
Each identity describes a relation between $G$-code and its dual codes with respect to a bilinear form given by the Euclidean inner product. 

In this paper, we define dual codes over arbitrary finite rings with respect to arbitrary bilinear forms, and introduce the notion of $G$-dual codes for $G$-codes. In addition, we present several generalizations of Atsumi's MacWilliams identity for $G$-codes over finite rings with respect to bilinear forms taking values in $\QZ$. 
More precisely, we prove the following theorems:

\begin{thm}{(Theorem \ref{hweMac})}

Let $M$ be a set of nondegenerate $\QZ$-valued $\Z$-bilinear forms on $V$ (as defined in \ref{bil}), $G$ a subgroup of the symmetric group of degree $n$, and $C$ a $G$-code. Then the $G$-Hamming weight enumerator $\hwe_G (\theta C) (x, y)$ of $\theta C$ (as defined in \ref{we}) satisfies 
\[
 \hwe_G (\theta C) (x, y) = \frac{1}{\abs{\theta (C^{\perp_M})}} \hwe_G (\theta (C^{\perp_M})) (x + (\abs{V} - 1) y, x - y). 
\]

\end{thm}

\begin{thm}{(Theorem \ref{fweMac})}

Let $M$ be a set of nondegenerate $\QZ$-valued $\Z$-bilinear forms on $V$ (as defined in \ref{bil}), $G$ a subgroup of the symmetric group of degree $n$, and $C$ a $G$-code. Then the $G$-full weight enumerator (as defined in \ref{we}) satisfies the following identity in $\C[V^t]$: 
\[
 \fwe_G ((\theta C)^{(\perp_G)_M}) 
= \frac{1}{\abs{\theta C}} \sum_{v \in \theta (V^n)} \sum_{u \in \theta C} \exp (2 \pi \sqrt{-1} \beta^n_G (v, u)) e_v. 
\]

\end{thm}

Gleason-type theorems 
(\cite{Gleason70, BMS72, MMS72, HO99, BDHO99}, and \cite[Chapter 19]{MS77}) 
are well-known and important results in coding theory and invariant theory. They state that the rings generated by the complete, Hamming, or other weight enumerators of self-dual codes satisfying certain algebraic conditions coincide with the invariant rings of specific associated groups. 
Nebe, Rains, and Sloane \cite{NRS04, NRS06} introduced the notions of full weight enumerators and Clifford--Weil groups. 
They showed that the full weight enumerators of self-dual and isotropic codes are invariant under the Clifford--Weil groups. 
Furthermore, they demonstrated that under quite general conditions, the space of invariant vectors of the Clifford--Weil groups is spanned by the full weight enumerators of self-dual and isotropic codes. 
Since the complete and the Hamming weight enumerators can be obtained as projections of the full weight enumerators, the transformation groups for the former can be drived from the Clifford--Weil groups. 
Based on these considerations, Nebe, Rains, and Sloane estimated the transformation groups under which the weight enumerators of self-dual codes are invariant to be the Clifford--Weil groups. 

There have been several generalizations and applications of their studies, 
such as those in \cite{NS23, GNR08, BOZ21, NRS01}. 
In this paper, we present a generalization of their theory to $G$-codes. 
We introduce some transformation groups of the $G$-full weight enumerators, 
thereby extending the definition of the Clifford--Weil groups. 
Moreover, we provide generalizations of Gleason-type theorems for $G$-full weight enumerators to show the invariance and the generation of these enumerators for $G$-self-dual and $G$-isotropic codes under the Clifford--Weil groups. 
More precisely, we prove the following theorems:

\begin{thm}{(Theorem \ref{parainv})}

Let $G$ be a subgroup of the symmetric group of degree $n$, $(M, \Phi)$ 
a quadratic pair on $R$ and $\rho_G = (\theta (V^n), (\rho_G)_M, (\rho_G)_{\Phi})$ a $G$-representation of $(M, \Phi)$. 
Then, the vectors $\fwe_G (\theta C)$, where $C$ ranges over all $G$-isotropic codes (with respect to $\rho_G$), span the invariant subspace of $\C [\theta (V^n)]$ under the action of $P (\rho_G)$. That is, we have 
\[
 \C [\fwe_G (\theta C) \mid \text{$C$ is a $G$-isotropic code (with respect to $\rho_G$)}] = \C [\theta (V^n)]^{P (\rho_G)}. 
\]
\end{thm}

\begin{thm}{(Theorem \ref{CWinv})}

Let $G$ be a subgroup of the symmetric group of degree $n$, $(R, M, \psi, \Phi)$ a form ring, and $\rho_G = (\theta (V^n), (\rho_G)_M, (\rho_G)_{\Phi}, \beta)$ a $G$-representation of $(R, M, \psi, \Phi)$. If $C$ is a $G$-code that is both $G$-self-dual and $G$-isotropic (with respect to $\rho_G$), then the $G$-full weight enumerator $\fwe_G (\theta C)$ is invariant under the Clifford--Weil group $\CW (\rho_G)$. 

\end{thm}

This paper is organized as follows. In Section \ref{bilin}, we provide an analogue of Hayden's theorem \cite{BHH81} for general bilinear forms. In Section \ref{G-code}, we generalize the definitions of $G$-codes and $G$-dual codes, and establish the MacWilliams identities for $G$-codes with respect to $G$-bilinear forms. In Section \ref{G-rep}, we present a generalization of representations of form rings and the Clifford--Weil groups \cite{NRS06} to the context of $G$-codes. Furthrmore, by applying the results from Section \ref{G-code}, we show that the $G$-full weight enumerators of $G$-self-dual and $G$-isotropic codes are invariant under the Clifford--Weil groups.

\section{Hayden's Theorem for bilinear forms over commutative rings} \label{bilin}

In this section, we define bilinear forms over commutative rings, and establish an analogue of \cite[Theorem 4.2]{BHH81} for these forms. 
We apply Theorem \ref{hay} to prove the MacWilliams identity 
in Theorem \ref{hweMac}. 
Our definition of bilinear forms over commutative rings follows the framework in \cite[Sections 1 and 3]{NRS06}.

\subsection{Bilinear forms over commutative rings}

We define bilinear forms over commutative rings. 

\begin{Def}

Let $k$ be a non-trivial unital commutative ring. 
We call $R$ a $k$-algebra if $R$ is a non-trivial unital ring, and there exists a ring homomorphism from $k$ to $R$ such that its image is containd in the center of $R$. 

\end{Def}

Throughout this paper, let $k$ be a non-trivial unital commutative ring, $R$ a $k$-algebra, 
$V$ a left $R$-module, $A$ a $k$-module, $n$ a positive integer, and $V^n$ the direct sum of $n$ copies of $V$. 

\begin{Def} \label{bil}

A $k$-module homomorphism 
\[
 \beta \colon V \otimes_k V \to A
\]
is called an $A$-valued $k$-bilinear form on $V$. 

\end{Def}

We denote the set of all $A$-valued $k$-bilinear forms on $V$ $\Bil_k (V, A)$. 

\begin{rem*}

Since $V$ is a left $R$-module, $\Bil_k (V, A)$ becomes a right $(R \otimes_k R)$-module by defining 
\[
 \beta (r \otimes_k s) (x \otimes_k y) = \beta (rx \otimes_k sy) 
\]
for $\beta \in \Bil(V, A)$, $x, y \in V$, and $r, s \in R$. 

\end{rem*}

We define the $A$-valued $k$-bilinear form on $V^n$ as the orthogonal sum of 
$n$ copies of an $A$-valued $k$-bilinear form on $V$. 

\begin{Def}

Let $\beta$ be an $A$-valued $k$-bilinear form on $V$. We define the $k$-module homomorphism $\beta^n \colon V^n \otimes_k V^n \to A$ by 
\[
 \beta^n (u \otimes_k v) = \sum^n_{i = 1} \beta (u_i \otimes_k v_i) 
\]
for $u = (u_1, \dots, u_n), \rv{v}{n} \in V^n$. 

\end{Def}

The pair $(V^n, \beta^n)$ is called the orthogonal sum of $n$ copies of $(V, \beta)$. 

\begin{ex}

Let $p$ be a prime number, $k$ the set of integers $\Z$, $\Q$ the set of rational numbers, and $R$ and $V$ the prime field $\F_p$ of order $p$. 
Given $A = \Q/\Z$, we define $\beta \colon \F_p \otimes_{\Z} \F_p \to \QZ$ to be 
\[
 \beta (x \otimes_{\Z} y) = \frac{1}{p} xy. 
\]
Then $\beta^n$ corresponds to the Euclidean inner product 
\[
 (u, v) = \sum^n_{i = 1} u_i v_i 
\] 
for $u = (u_1, \dots, u_n), \rv{v}{n} \in \F_p^n$. 
\end{ex}

In light of this example, we can regard Theorem \ref{hay} as an analogue of \cite[Theorem 4.2]{BHH81}. 

\begin{Def}

Let $C$ be an $R$-submodule of $V^n$ and $M \subset \Bil_k (V, A)$. We call 
\[
 C^{\perp_M} = \{ v \in V^n \mid \beta^n (v \otimes_k u) = 0, \quad \text{for all $\beta \in M$ and all $u \in C$} \} 
\]
the dual of $C$ (with respect to $M$). 

If $C \subset C^{\perp_M}$, we say that $C$ is self-orthogonal (with respect to $M$). 

If $C = C^{\perp_M}$, we say that $C$ is self-dual (with respect to $M$). 

\end{Def}

\subsection{Modules with group actions}

We define the group action on codes. 
Throughout this paper, we fix the action of a group $G$ on $V^n$ by the following definition \ref{grp}.

\begin{Def} \label{grp}

Let $\rv{u}{n} \in V^n$ and $g \in \GL(k^n)$. 
We define $v = g \cdot u \in V^n$ to be the linear transformation induced by $g^{-1}$, i.e., 
\[
 v = (v_1, \dots, v_n) \ceq (u_1, \dots, u_n) g^{-1}. 
\]
More precisely, if $g^{-1} = (g_{ij})^n_{i,j = 1}$, then each component $v_i$ is defined by 
\[
 v_i = u_1 g_{1i} + \cdots + u_n g_{ni}. 
\]

\end{Def}

This transformation defines a left action on $V^n$ of $\GL(k^n)$. 

\begin{rem*}

Let $G$ be a finite subgroup of $\GL(k^n)$ and $R[G]$ the group ring of $G$ over $R$. 
By defining 
\[
 \theta u = \sum_g \theta_g g \cdot u 
\]
for any $\theta = \sum_g \theta_g g \in R[G]$ and $u \in V^n$, 
we can regard $V^n$ as a left $R[G]$-module. 

\end{rem*}

In this paper, we define $G$ to be a finite subgroup of $\GL(k^n)$ and $R[G]$ to be the group ring of $G$ over $R$. If $C$ is an $R$-submodule of $V^n$ and $G$ acts on $C$ defined in \ref{grp}, we call $C$ an $R[G]$-submodule of $V^n$. 

For an $R[G]$-submodule $C \subset V^n$ and $\theta \in R[G]$, 
we define 
\[
 \theta C = \theta (C) \ceq \{ \theta u \in C \mid u \in C \}. 
\]

\subsection{Hayden's Theorems for bilinear forms}

We provide an analogue of \cite[Theorem 4.2]{BHH81} for bilinear forms over commutative rings. In what follows, for the sake of simplicity, we write $\otimes_k$ and $\Bil_k (V, A)$ simply as $\otimes$ and $\Bil (V, A)$, respectively. 

\begin{lem}

Let $\beta \in \Bil (V, A)$. %be an $A$-valued $k$-bilinear form on $V$. 
For any $g \in \GL (k^n)$, we write $g^T \in \GL (k^n)$ for its transpose. 
Then 
\[
 \beta^n (u \otimes (g \cdot v)) = \beta^n ((g^T \cdot u) \otimes v) 
\]
for all $u, v \in V^n$. 

\end{lem}

\begin{proof}

Let $u = (u_1, \dots, u_n), \rv{v}{n} \in V^n$, and $g^{-1} = (g_{ij})^n_{i, j = 1} \in \GL (k^n)$. Then it holds that  
\begin{align*}
  \beta^n (u \otimes (g \cdot v)) 
& = \sum^n_{i = 1} \beta \left( u_i \otimes \left( \sum^n_{j = 1} v_j g_{ji} \right) \right) \\
& = \sum^n_{i = 1} \sum^n_{j = 1} \beta (u_i \otimes v_j g_{ji}) \\
%& = \sum^n_{j = 1} \sum^n_{i = 1} \beta (u_i \otimes v_j g_{ji}) \\
& = \sum^n_{j = 1} \sum^n_{i = 1} \beta (u_i \otimes g_{ji} v_j) \\
& = \sum^n_{j = 1} \sum^n_{i = 1} \beta (u_i g_{ji} \otimes v_j) \\
& = \sum^n_{j = 1} \beta \left( \left( \sum^n_{i = 1} u_i g_{ji} \right) \otimes v_j \right) \\
& = \beta^n (u (g^{-1})^T \otimes v) \\
& = \beta^n (u (g^T)^{-1} \otimes v) \\
& = \beta^n ((g^T \cdot u) \otimes v). 
\end{align*}
\end{proof}

\begin{thm} \label{hay}

Let $M \subset \Bil (V, A)$ and let $C$ be an $R[G]$-submodule of $V^n$. 
If $\theta \in R[G]$ is an idempotent, i.e., $\theta^2 = \theta$, then 
\[
 (\theta C)^{\perp_M} = \ker \theta^T \oplus \theta^T (C^{\perp_M}). 
\]

\end{thm}

\begin{proof}

We prove $(\theta C)^{\perp_M} = \ker \theta^T + \theta^T (C^{\perp_M})$ 
and $\ker \theta^T \cap \theta^T (C^{\perp_M}) = \{ 0 \}$. 

Letting $u \in C$, $v \in C^{\perp_M}$, $w \in \ker \theta^T$ and $\beta \in M$, we have 
\begin{align*}
 \beta^n (\theta^T v \otimes \theta u) 
& = \beta^n (v \otimes \theta^2 u) = \beta^n (v \otimes \theta u) = 0, \\
 \beta^n (w \otimes \theta u) 
& = \beta^n (\theta^T w \otimes u) = \beta^n (0 \otimes u) = 0. 
\end{align*}
Hence this means that $(\theta C)^{\perp_M} \supset \ker \theta^T + \theta^T (C^{\perp_M})$. 

Conversely, let $v \in (\theta C)^{\perp_M}$. For any $u \in C$ and $\beta \in M$, we obtain 
\[
 0 = \beta^n (v \otimes \theta u) = \beta^n (\theta^T v \otimes u), 
\]
which implies $\theta^T v \in C^{\perp_M}$. Consequently, we can decompose $v$ as
\[
 v = v - \theta^T v + \theta^T v = (v - \theta^T v) + \theta^T (\theta^T v) \in \ker \theta^T + \theta^T (C^{\perp_M}). 
\]
Thus this yields $(\theta C)^{\perp_M} \subset \ker \theta^T + \theta^T (C^{\perp_M})$. 

Suppose $v \in \ker \theta^T \cap \theta^T (C^{\perp_M})$. Then there exists $u \in C^{\perp_M}$ such that $v = \theta^T u$. It follows that 
\[
 0 = \theta^T v = \theta^T (\theta^T u) = \theta^T u = v. 
\]
Therefore, we see that $\ker \theta^T \cap \theta^T (C^{\perp_M}) = \{ 0 \}$.  

\end{proof}

\begin{cor} \label{Snhay}

Let $M \subset \Bil (V, A)$ and let $C$ be an $R[G]$-submodule of $V^n$. Define 
\[
 G^T = \{ g^T \in \GL(k^n) \mid g \in G \} 
\]
and 
\[
\theta = \frac{1}{\abs{G}} \sum_{g \in G} g. 
\]
Suppose that $G^T = G$ and the order $\abs{G}$ of $G$ is a unit in $R$. Then 
\[
 (\theta C)^{\perp_M} = \ker \theta \oplus \theta (C^{\perp_M}). 
\] 

\end{cor}

\begin{proof}
From 
\[
 \theta^2 = \left( \frac{1}{\abs{G}} \sum_{g \in G} g \right)^2 = \frac{1}{\abs{G}^2} \sum_{g, h \in G} g h 
 = \frac{1}{\abs{G}} \sum_{g \in G} g = \theta, 
\]
we observe that $\theta$ is an idempotent and satisfies the hypothesis of Theorem \ref{hay}. In addition, the assumption $G^T = G$ implies that 
\[
  \theta^T = \frac{1}{\abs{G}} \sum_{g \in G} g^T 
= \frac{1}{\abs{G}} \sum_{g \in G} g = \theta. 
\]
Thus applying Theorem \ref{hay}, we see that $(\theta C)^{\perp_M} = \ker \theta \oplus \theta (C^{\perp_M})$. 

\end{proof}

\section{$G$-codes and MacWilliams identities} \label{G-code}

In this section, we present a generalization of the definition of codes with a group action \cite{CM23}, and define dual codes as the orthogonal complements with respect to bilinear forms \cite{NRS06}. In particular, we follow the approach in \cite{CM23} to define $G$-codes, $G$-bilinear forms, and the $G$-Hamming weight. 

We identify the symmetric group with its permutation representation, and define $G$ as a subgroup of the symmetric group of degree $n$. 
(Note in advance that the term "representation" doesn't mean linear representation in Section \ref{G-rep}. ) 

We set $k = \Z$ and simply refer to $\Z$-bilinear forms as bilinear form. 
Furthermore, we define $[n] \ceq \{1, \dots, n \}$ and $ \theta = \frac{1}{\abs{G}} \sum_{g \in G} g$. 
We assume that the order $\abs{G}$ of $G$ is a unit in $R$. 

\subsection{$G$-codes}

We define $G$-codes and relevant concepts. 

\begin{Def}
 
Let $G \alpha_1, \dots, G \alpha_t$ be the orbits of the coordinates of $V^n$ 
under the action of $G$, where $t$ denotes the number of these orbits, and $\{ \alpha_1, \dots, \alpha_t \} \subset [n]$ is a complete system of representatives for these orbits. That is, we define 
\[
 G \alpha_i = \{ g^{-1} (\alpha_i) \mid g \in G \} 
\]
for each $i \in [t]$, such that 
\[
 G \alpha_i \cap G \alpha_j = \emptyset 
\]
whenever $i \neq j$. 

The size $\abs{G \alpha_i}$ of each orbit $G \alpha_i$ is denoted by $m_{\alpha_i}$. 

\end{Def}

For $g \in G$ and $\rv{v}{n} \in V^n$, we define the action of $g$ on $v$ by 
\[
 g \cdot v \ceq v g^{-1} = ( v_{g^{-1} (1)}, \dots, v_{g^{-1} (n)}). 
\]

\begin{Def}

For each $i \in [t]$, we define
$\overline{G \alpha_i} = (y_1, \dots, y_n) \in V^n$ as follows: 
\begin{equation*}
y_j = 
\begin{cases}
 1 & j \in G \alpha_i, \\
 0 & j \notin G \alpha_i. 
\end{cases}
\end{equation*}

\end{Def}

\begin{Def}

Any element $\rv{u}{n} \in \theta (V^n)$ can be uniquely expressed as 
\[
 \rv{u}{n} = \sum^t_{i = 1} u_{\alpha_i} \overline{G \alpha_i}. 
\]
for some $u_{\alpha_i} \in V$. 

We then define the corresponding vector $u_G \in V^t$ by 
$u_G = (u_{\alpha_1}, \dots, u_{\alpha_t})$. 

\end{Def}

Note that the map 
\[
( - )_G \colon \theta (V^n) \to V^t; \ \rv{u}{n} \mapsto u_G = (u_{\alpha_1}, \dots, u_{\alpha_t}) 
\]
is an $R$-module isomorphism. 

By using these definitions, we now define $G$-bilinear forms and the $G$-Hamming weight. 

\begin{Def}

Let $A = \Q/\Z$ and $\beta \in \Bil (V, \QZ)$. 

We define the $\Z$-module homomorphism $\beta^n_G \colon \theta (V^n) \otimes \theta (V^n) \to \Q/\Z$ by 
\[
 \beta^n_G (u \otimes v) \ceq \beta^t (u_G \otimes v_G) = \sum^t_{i = 1} \beta (u_{\alpha_i} \otimes v_{\alpha_i}) 
\]
for any $\rv{u}{n}$, $\rv{v}{n} \in \theta (V^n)$. 

\end{Def}

We call this homomorphism the $\Q/\Z$-valued $G$-bilinear form on $\theta (V^n)$ 
and write the set of such forms as $\Bil (\theta (V^n), \Q/\Z)_G$. Note that if $G = \{ e \}$, then $\beta^n_{\{ e \}} = \beta^n$ and $\Bil (\theta (V^n), \Q/\Z)_{\{ e \}} \subset \Bil (V^n, \Q/\Z)$. 

\begin{Def}

We define the map $\wt_G \colon \theta (V^n) \to \{0, 1, \dots, t \}$ by 
\[
 \wt_G (u) = \abs{ \{ i \in [t] \mid u_{\alpha_i} \neq 0 \} } 
\]
for all $\rv{u}{n} \in \theta (V^n)$. 

We call $\wt_G (u)$ the $G$-Hamming weight of $u$. 

\end{Def}

It is clear that if $G = \{ e \}$, then $\wt_{\{ e \}}$ coincides with the Hamming weight $\wt$. 

In what follows, we assume that $R$ and $V$ are finite sets, and we fix $A = \Q/\Z$. 

\begin{Def}

An $R[G]$-submodule $C$ of $V^n$ is called a $G$-code. 

Let $M \subset \Bil (V, \Q/\Z)$ and let $C$ be a $G$-code. 
We define the $G$-dual code of $C$ (with respect to $M$) by 
\[
 (\theta C)^{(\perp_G)_M} = \{ v \in \theta (V^n) \mid \beta^n_G (v \otimes u) = 0 \quad 
\text{for all $\beta \in M$ and all $u \in \theta C$} \}. 
\]

Furthermore, we say that $C$ is $G$-self orthogonal (with respect to $M$) 
if $\theta C \subset (\theta C)^{(\perp_G)_M}$, and $G$-self dual (with respect to $M$) if $\theta C = (\theta C)^{(\perp_G)_M}$. 

\end{Def}

%We refer to $R[G]$-submodules of finite modules $V^n$ on finite rings $R$ as $G$-codes. 
Note that if $G = \{ e \}$, then $(\theta C)^{(\perp_{\{ e \}})_M}$ coincides with the dual code $C^{\perp_M}$. 

%We describe algebraic properties like the MacWilliams identities for a $G$-code in Section \ref{MWi}. 

\subsection{$G$-weight enumerators}

We define the $G$-weight enumerators of $G$-codes. Moreover, we drive the MacWilliams identities, which relate the $G$-weight enumerators of $G$-codes to that of $G$-dual codes. 
%We also show the lemma which we use to prove MacWilliams identity \ref{hweMac}. 
We follow the approach in \cite{NRS06} for the definitions of these weight enumerators. 

Let $\C$ be complex numbers. 

\begin{Def} \label{we}
 
Let $C$ be a $G$-code. The $G$-Hamming weight enumerator of $\theta C$ is defined by
\[
 \hwe_G (\theta C) (x, y) = \sum_{u \in \theta C} x^{t - \wt_{G} (u)} y^{\wt_{G} (u)} \ \in \C[x,y]. 
\]

The $G$-full weight enumerator of $\theta C$ is defined by 
\[
 \fwe_G (\theta C) = \sum_{u \in \theta C} e_u \ \in \C[V^t]. 
\] 

Furthermore, we define $\cwe_G (\theta C) (x_v : v \in V) \in \C[x_v \mid v \in V]$ as the image of $\fwe_G (\theta C)$ under the projection 
\[
 \C[V^t] \to \C[x_v \mid v \in V]; \ e_v = e_{(v_{\alpha_1}, \dots, v_{\alpha_t})} \mapsto \prod^t_{i = 1} x_{v_{\alpha_i}}. 
\] 

We call $\cwe_G (\theta C) (x_v : v \in V)$ the $G$-complete weight enumerator of $\theta C$. 

\end{Def}

Note that if $G = \{ e \}$, we define $\hwe_{\{ e \}} = \hwe$, $\fwe_{\{ e \}} = \fwe$, and $\cwe_{\{ e \}} = \cwe$, respectively. 

By substituting $x_0 = x$ and $x_a = y$ for all $0 \neq a \in V$ into $\cwe_G (\theta C) (x_v : v \in V)$, we obtain the $G$-Hamming weight enumerator: 
\[
 \cwe_G (\theta C) (x_0, y : 0 \neq a \in V) = \hwe_G (\theta C) (x, y). 
\]

We refer to \cite{CM23} for the following definitions. 

\begin{Def}
 
The $G$-orbit length matrix $M_G$ is the $n \times n$ diagonal matrix $M_G = \mathrm{diag} (l_1, \dots, l_n)$, where the $j$-th diagonal entry is given by 
\[
 l_j = m_{\alpha_i} \quad 
(\text{for $j \in G \alpha_i$, $i \in [t]$, and $\alpha_i \in [n]$}). 
\]

\end{Def}

\begin{lem} \label{dual}

Let $M \subset \Bil (V, \Q/\Z)$. 
If $C$ is a $G$-code, then 
\[
 (\theta C)^{(\perp_G)_M} = \theta (C^{\perp_M}) M_G. 
\]

\end{lem}

\begin{proof}

By Corollary \ref{Snhay}, we have $\theta (C^{\perp_M}) \subset (\theta C)^{\perp_M}$. 
Thus for $u = \sum^t_{i = 1} u_{\alpha_i} \overline{G \alpha_i} \in \theta C$, $\beta \in M$ and $v = \sum^t_{i = 1} v_{\alpha_i} \overline{G \alpha_i} \in \theta (C^{\perp_M})$, it follows that 
\begin{align*}
0 = \beta^n (v \otimes u) 
& = \sum^n_{i = 1} \beta (v_i \otimes u_i) \\
& = \sum^t_{i = 1} m_{\alpha_i} \beta (v_{\alpha_i} \otimes u_{\alpha_i}) \\
& = \sum^t_{i = 1} \beta (m_{\alpha_i} v_{\alpha_i} \otimes u_{\alpha_i}) \\
& = \sum^t_{i = 1} \beta (v'_{\alpha_i} \otimes u_{\alpha_i}) \\
& = \beta^n_G (v' \otimes u), 
\end{align*}
which 
\[
 v' \ceq \sum^t_{i = 1} v'_{\alpha_i} \overline{G \alpha_i} \ceq 
\sum^t_{i = 1} m_{\alpha_i} v_{\alpha_i} \overline{G \alpha_i} = v M_G. 
\]
Hence this yields $(\theta C)^{(\perp_G)_M} \supset \theta (C^{\perp_M}) M_G$. 

Conversely, let $u = \sum^t_{i = 1} u_{\alpha_i} \overline{G \alpha_i} \in (\theta C)^{(\perp_G)_M}$ and $v = \sum^t_{i = 1} v_{\alpha_i} \overline{G \alpha_i} \in \theta C$. 
Since $\abs{G}$ in a unit $\in R$ and $\abs{G \alpha_i} = m_{\alpha_i}$, We can assume $m_{\alpha_i}$ is a unit $\in R$. Therefore we can define 
\[
 u' \ceq \sum^t_{i = 1} \frac{u_{\alpha_i}}{m_{\alpha_i}} \overline{G \alpha_i} \in \theta (V^n), 
\]
and we see that 
\begin{align*}
\beta^n (u' \otimes v) 
& = \sum^n_{i = 1} \beta (u'_i \otimes v_i) \\
& = \sum^t_{i = 1} m_{\alpha_i} \beta (u'_{\alpha_i} \otimes v_{\alpha_i}) \\
& = \sum^t_{i = 1} m_{\alpha_i} \beta (\frac{u_{\alpha_i}}{m_{\alpha_i}} \otimes v_{\alpha_i}) \\
& = \sum^t_{i = 1} \beta (u_{\alpha_i} \otimes v_{\alpha_i}) \\
& = \beta^n_G (u \otimes v) = 0. 
\end{align*}
Thus, applying Corollary \ref{Snhay}, it follows that $u' \in (\theta C)^{\perp_M} = \ker \theta \oplus \theta (C^{\perp_M})$. 
In addition we observe that $u' \in \ker \theta \ \Leftrightarrow \ u' = 0$ because 
\begin{align*}
 \theta u'
& = \theta \left( \sum^t_{i = 1} \frac{u_{\alpha_i}}{m_{\alpha_i}} \overline{G \alpha_i} \right) \\
& = \left( \frac{1}{\abs{G}} \sum_{g \in G} g \right) 
\left( \sum^t_{i = 1} \frac{u_{\alpha_i}}{m_{\alpha_i}} \overline{G \alpha_i} \right) \\
& = \frac{1}{\abs{G}} \sum_{g \in G} g \cdot \left( \sum^t_{i = 1} \frac{u_{\alpha_i}}{m_{\alpha_i}} \overline{G \alpha_i} \right) \\
& = \frac{1}{\abs{G}} \sum^t_{i = 1} \frac{u_{\alpha_i}}{m_{\alpha_i}} \sum_{g \in G} \left( g \cdot \overline{G \alpha_i} \right) \\
& = \frac{1}{\abs{G}} \sum^t_{i = 1} \frac{u_{\alpha_i}}{m_{\alpha_i}} \sum_{g \in G} \overline{G \alpha_i} \\
& = \sum^t_{i = 1} \frac{u_{\alpha_i}}{m_{\alpha_i}} \overline{G \alpha_i} = u'. 
\end{align*}
Hence we obtain $u = u' M_G \in \theta (C^{\perp_M}) M_G$. 
This implies that $(\theta C)^{(\perp_G)_M} \subset \theta (C^{\perp_M}) M_G$. 

\end{proof}

Equipped with this lemma, we are now able to establish the MacWilliams-type identities that relate the $G$-Hamming weight enumerators of $G$-codes to that of its $G$-dual codes.

\subsection{MacWilliams identities} \label{MWi}

We discuss the relationship between $G$-codes and their $G$-dual codes, specifically through the MacWilliams identities. These identities provide fundamental relations between the weight enumerators of codes and their duals. 

First, we establish the MacWilliams identity for the $G$-Hamming weight enumerator. 
To this end, we employ the following lemma, which follows from \cite[Example 2.2.5]{NRS06}. 

\begin{lem} \label{Mac}

Let $M$ be a set of nondegenerate bilinear forms in $\Bil (V, \Q/\Z)$, $N$ a positive number, and $C$ a finite $R$-submodule of $V^N$. For $u \in V^N$, define $f (u) \in \C[x,y]$ by 
\[
 f (u) =x^{N - \wt (u)} y ^{\wt (u)}. 
\]
Then the following identity holds: 
\[
 \sum_{u \in C} f (u) = \frac{1}{\abs{(C)^{\perp_M}}} \sum_{v \in C^{\perp_M}} \hat{f} (v), 
\]
where 
\[
 \hat{f} (v) = (x - y)^{\wt (v)} (x + (\abs{V} -1) y)^{N - \wt (v)}. 
\]

\end{lem}

%Above the lemma lead to an identity. 
We establish the first MacWilliams identity in this paper. 

\begin{lem} \label{GMac}

Let $M$ be a set of nondegenerate bilinear forms in $\Bil (V, \Q/\Z)$ and let $C$ be a $G$-code. 
Then the $G$-Hamming weight enumerator $\hwe_G (\theta C) (x, y)$ satisfies the identity: 
\[
 \hwe_G (\theta C) (x, y) = \frac{1}{\abs{(\theta C)^{(\perp_G)_M}}} \hwe_G ((\theta C)^{(\perp_G )_M}) (x + (\abs{V} - 1) y, x - y). 
\]
\end{lem}

\begin{proof}
Recall the $R$-module isomorphism $\theta (V^n) \to V^t$ defined by $u \mapsto u_G$. Under this identification, the $G$-Hamming weight $\wt_G (u)$ of $u \in \theta (V^n)$ satisfies  
\[
 \wt_G (u) = \abs{ \{ i \in [t] \mid u_{\alpha_i} \neq 0 \} } = \wt (u_G). 
\]
Thus, by applying Lemma \ref{Mac} with $N = t$, %For $u \in \theta (V^n) \cong V^t$, since $f (u) \in \C[x,y]$ is written by \[ f (u) =x^{t - \wt (u_G)} y ^{\wt (u_G)} = x^{t - \wt_G (u)} y ^{\wt_G (u)}, \]
and noting that the dual $R$-module of $\{ u_G \mid u \in \theta C \} \subset V^t$ coincides with $(\theta C)^{(\perp_G)_M}$, we have 
\begin{align*}
 \hwe_G (\theta C) (x, y) & = \sum_{u \in \theta C} f (u) \\
& = \frac{1}{\abs{(\theta C)^{(\perp_G)_M}}} 
\sum_{v \in (\theta C)^{(\perp_G)_M}} \hat{f} (v) \\
& = \frac{1}{\abs{(\theta C)^{(\perp_G)_M}}} 
\hwe_G ((\theta C)^{(\perp_G)_M}) (x + (\abs{V} - 1) y, x - y). 
\end{align*}

\end{proof}

This identity is referred to as the MacWilliams identity for the $G$-Hamming weight enumerator. 

Next, we establish the second MacWilliams identity in this paper, by employing Lemma \ref{dual} and Lemma \ref{GMac}. 

\begin{thm} \label{hweMac}

Let $M$ be a set of nondegenerate bilinear forms in $\Bil (V, \Q/\Z)$, $G$ a subgroup of the symmetric group of degree $n$, and $C$ a $G$-code. Then the $G$-Hamming weight enumerator $\hwe_G (\theta C) (x, y)$ satisfies 
\[
 \hwe_G (\theta C) (x, y) = \frac{1}{\abs{\theta (C^{\perp_M})}} \hwe_G (\theta (C^{\perp_M})) (x + (\abs{V} - 1) y, x - y). 
\]

\end{thm}

\begin{proof}

By Lemma \ref{dual}, the $G$-dual code is given by $(\theta C)^{(\perp_G)_M} = \theta (C^{\perp_M}) M_G$. Substituting this into Lemma \ref{GMac}, we have 
\begin{align*}
 \hwe_G (\theta C) (x, y) 
& = \frac{1}{\abs{(\theta C)^{(\perp_G)_M}}} \hwe_G ((\theta C)^{(\perp_G)_M}) (x + (\abs{V} - 1) y, x - y) \\
& = \frac{1}{\abs{\theta (C^{\perp_M}) M_G}} \hwe_G (\theta (C^{\perp_M}) M_G) (x + (\abs{V} - 1) y, x - y). 
\end{align*}
For any $u = \sum^t_{i = 1} u_{\alpha_i} \overline{G \alpha_i} \in \theta (C^{\perp_M})$, since each $m_{\alpha_i}$ is a unit in $R$, it follows that 
\begin{align*}
\wt_G (u) 
& = \abs{\{ i \mid u_{\alpha_i} \neq 0 \}} \\
& = \abs{\{ i \mid m_{\alpha_i} u_{\alpha_i} \neq 0 \}} \\
& = \wt_G (u M_G). 
\end{align*}
Hence we conclude that 
\begin{align*}
\hwe_G (\theta C) (x, y) 
& = \frac{1}{\abs{\theta (C^{\perp_M}) M_G}} \hwe_G (\theta (C^{\perp_M}) M_G) (x + (\abs{V} - 1) y, x - y) \\
& = \frac{1}{\abs{\theta (C^{\perp_M})}} \hwe_G (\theta (C^{\perp_M})) (x + (\abs{V} - 1) y, x - y). 
\end{align*}
\end{proof}

We can call this theorem the $G$-MacWilliams identity of the $G$-Hamming weight enumerator. 

Finally, we present the MacWilliams identity for the $G$-full weight enumerator. 

\begin{thm} \label{fweMac}

Let $M$ be a set of nondegenerate bilinear forms in $\Bil (V, \Q/\Z)$, $G$ a subgroup of the symmetric group of degree $n$, and $C$ a $G$-code. Then the $G$-full weight enumerator satisfies the following identity in $\C[V^t]$: 
\[
 \fwe_G ((\theta C)^{(\perp_G)_M}) 
= \frac{1}{\abs{\theta C}} \sum_{v \in \theta (V^n)} \sum_{u \in \theta C} \exp (2 \pi \sqrt{-1} \beta^n_G (v, u)) e_v. 
\]

\end{thm}

\begin{proof}

Since $\theta C$ is an Abelian group, given $v \in \theta (V^n)$, the map 
\[
 \theta C \to \C^*; \ u \mapsto \exp (2 \pi \sqrt{-1} \beta^n_G (v, u)) 
\]
defines a character of $\theta C$. By the orthogonality relations of characters, we have 
\begin{align*}
 \sum_{u \in \theta C} \exp (2 \pi \sqrt{-1} \beta^n_G (v, u)) 
& = \frac{\abs{\theta C}}{\abs{\theta C}} \sum_{u \in \theta C} \exp (2 \pi \sqrt{-1} \beta^n_G (v, u)) \\
& = 
\begin{cases}
 \abs{\theta C} & \quad (v \in (\theta C)^{(\perp_G)_M}) \\
 0 & \quad (v \notin (\theta C)^{(\perp_G)_M}). 
\end{cases}
\end{align*}
Thus we obtain 
\begin{align*}
& \sum_{v \in \theta (V^n)} \sum_{u \in \theta C} \exp (2 \pi \sqrt{-1} \beta^n_G (v, u)) e_v \\
& = \sum_{v \in (\theta C)^{(\perp_G)_M}} \sum_{u \in \theta C} \exp (2 \pi \sqrt{-1} \beta^n_G (v, u)) e_v 
+ \sum_{v \notin (\theta C)^{(\perp_G)_M}} \sum_{u \in \theta C} \exp (2 \pi \sqrt{-1} \beta^n_G (v, u)) e_v \\
& =  \sum_{v \in (\theta C)^{(\perp_G)_M}} \abs{\theta C} e_v 
+ \sum_{v \notin (\theta C)^{(\perp_G)_M}} 0 \ e_v \\
& = \abs{\theta C} \sum_{v \in (\theta C)^{(\perp_G)_M}} e_v + 0 \\
& = \abs{\theta C} \fwe_G ((\theta C)^{(\perp_G)_M}). 
\end{align*}

\end{proof}

This theorem implies that the $G$-full weight enumerator of a $G$-self-dual code is invariant under this transformation. 
We will further investigate the properties of this transformation and the resulting invariant theory in Section \ref{G-rep}.

\section{$G$-representations of form rings and the Clifford--Weil group} \label{G-rep}

In this section, we develop the $G$-analogue of the representations of form rings and explore their associated algebraic structures. In addition, we demonstrate that the $G$-full weight enumerators of $G$-self-dual isotropic codes are invariant under the actions of the Clifford--Weil groups associated with these representations. 

Throughout this section, we assume that $G$ is a subgroup of the symmetric group of degree $n$.

\subsection{$G$-quadratic maps}

Following \cite{NRS06}, we define $G$-quadratic maps and $G$-isotropic codes.  In particular, our definition of quadratic maps is based on \cite[Chapter 1]{NRS06}. 

\begin{Def}

A map $\phi \colon V \to \QZ$ is called a $\QZ$-valued quadratic map on $V$ if it satisfies 
\[
 \phi (u + v + w) + \phi (u) + \phi (v) + \phi (w) = \phi (u + v) + \phi (v + w) + \phi (w + u) + \phi (0) 
\]
for any $u, v, w \in V$. 

\end{Def}

The set of all $\QZ$-valued quadratic maps on $V$ is denoted by $\QVQZ$. Furthermore, let $\Qz \ceq \{ \phi \in \QVQZ \mid \phi (0) = 0 \}$. We note that both $\QVQZ$ and $\Qz$ are $\Z$-modules. 

\begin{Def}

Let $\phi$ be a $\QZ$-valued quadratic map on $V$. We define the $\QZ$-valued quadratic map $\phi^n \colon V^n \to \QZ$ on $V^n$ as the direct sum of $n$ copies of $\phi$. That is, for all $\rv{u}{n} \in V^n$, we define 
\[
 \phi^n (u) = \sum^n_{i = 1} \phi (u_i). 
\]

%The pair $(V^n, \phi^n)$ is called the direct sum of $n$ copies of $(V, \phi)$.  

\end{Def}

%We define the $G$-quadratic map using a quadratic map. 

\begin{Def}

Let $\phi \in \QVQZ$. We define the map $\phi^n_G \colon \theta (V^n) \to \QZ$ by 
\[
 \phi^n_G (u) \ceq \phi^t (u_G) = \sum^t_{i = 1} \phi (u_{\alpha_i}) 
\]
for all $u = \sum^t_{i = 1} u_{\alpha_i} \overline{G \alpha_i} \in \theta (V^n)$. 

We call $\phi^n_G$ a $\QZ$-valued $G$-quadratic map on $\theta (V^n)$. 

\end{Def}

We write $\QnG$ for the set of all $\QZ$-valued $G$-quadratic maps on $\theta (V^n)$. Moreover, let $\QG \ceq \{ \phi \in \QnG \mid \phi (0) = 0 \}$. 

\begin{Def}

Let $M \subset \Bil (V, \QZ)$, $\Phi \subset \Quad_0 (V, \QZ)$ and let $C$ be a $G$-code. 
We say that $C$ is $G$-isotropic (with respect to $(M, \Phi)$) if $C$ is $G$-self-orthogonal (with respect to $M$) and satisfies 
\[
 \phi^n_G (u) = 0 
\]
for all $\phi \in \Phi$ and all $u \in \theta C$. 

\end{Def}

Note that if $G = \{ e \}$, then $\phi^n_{\{ e \}} = \phi^n$. In this case, $\Quad (\theta (V^n), \QZ)_{\{ e \}}  \subset \Quad (V^n, \QZ)$ and $\Quad_0 (\theta (V^n), \QZ)_{\{ e \}}  \subset \Quad_0 (V^n, \QZ)$.

\subsection{$G$-self-dual codes}

Analogously, by utilizing the orthogonal sum and the map $(-)_G$, we can define $G$-representations for twisted $R$-modules, twisted rings, quadratic pairs, and form rings. 
For instance, the $G$-representations of twisted $R$-modules are given by the following definitions. 

\begin{Def}

Let $M$ be a right $(R \otimes R)$-module. We say that $M$ is a twisted $R$-module if there exists an automorphism $\tau \colon M \to M$ with $\tau^2 = \id_M$ and 
\[
 \tau (m (r \otimes s)) = \tau (m) (s \otimes r) 
\] 
for $m \in M$ and $r, s \in R$. 
\end{Def}

\begin{Def}

Let $M$ be a twisted $R$-module. A $G$-representation of $M$ is a pair of $\rho_G = (\theta (V^n), (\rho_G)_M)$, where 
\[
 (\rho_G)_M \colon M \to \Bil (\theta (V^n), \QZ)_G 
\]
is a twisted $R$-module homomorphism. Specifically, $(\rho_G)_M$ is a right $(R \otimes R)$-module homomorphism satisfying 
\[
 (\rho_G)_M (\tau (m)) (u \otimes v) = (\rho_G)_M (m) (v \otimes u) 
\]
for any $m \in M$ and $u, v \in \theta (V^n)$. 

\end{Def}

Since $R$ and $V$ are finite sets and $A = \QZ$, $\rho_G$ is indeed a finite $G$-representation. Hereafter, we shall simply refer to it as a $G$-representation. 

While we have defined $G$-representations via the map $(-)_G$, it is important to note that $G$-representations can be decomposed into the orthogonal sum of $t$ copies of  the underlying representations. 

\begin{rem*}

Let $(R, M, \psi, \Phi)$ be a form ring and let 
$\rho_G = (\theta (V^n), (\rho_G)_M, (\rho_G)_{\Phi}, \beta)$ be a $G$-representation of $(R, M, \psi, \Phi)$. Then there exists a $\{ e \}$-representation 
$\rho' \ceq  \rho_{\{ e \}}' = (V, \rho'_M, \rho'_{\Phi},\beta')$ with underlying module $V$, such that $\rho_G$ is the orthogonal sum of $t$ copies of $\rho'$. 
That is, $\rho'$ satisfies the following conditions:  
\begin{align*}
 (\rho_G)_M (m) (u \otimes v)  
&= \sum^t_{i = 1} \rho'_M (m) (u_{\alpha_i} \otimes v_{\alpha_i}), \\
 (\rho_G)_{\Phi} (\phi) (u)  
&= \sum^t_{i = 1} \rho'_{\Phi} (\phi) (u_{\alpha_i}), \\
 \beta (u \otimes v) 
& = \sum^t_{i = 1} \beta' (u_{\alpha_i} \otimes v_{\alpha_i}), 
\end{align*}
for $m \in M$, $\phi \in \Phi$, and 
$u = \sum^t_{i = 1} u_{\alpha_i} \overline{G \alpha_i}, 
v = \sum^t_{i = 1} v_{\alpha_i} \overline{G \alpha_i} \in \theta (V^n)$. 

\end{rem*}

\begin{Def}

Let $(R, M, \psi, \Phi)$ be a form ring and let 
$\rho_G = (\theta (V^n), (\rho_G)_M, (\rho_G)_{\Phi}, \beta)$ be 
a $G$-representation of $(R, M, \psi, \Phi)$. A $G$-code $C$ is called a $G$-self-dual isotropic code (with respect to $\rho_G$) if the following conditions holds: 
\begin{align*}
 & \theta C = (\theta C)^{( \perp_G )_{\rho_G}} \ceq 
 \{ v \in \theta (V^n) \mid \beta (v \otimes u) = 0 \quad \text{for all $u \in \theta C$} \}, \\
 & (\rho_G)_{\Phi} (\phi) (u) = 0 \quad \text{for all $u \in \theta C$ and all $\phi \in \Phi$}. 
\end{align*}
\end{Def}

\begin{rem*}

By the definition of the $G$-representation $\rho_G$, it follows that for any $v \in \theta (V^n)$, 
\[
 v \in (\theta C)^{(\perp_G)_{\rho_G}} 
\]
if and only if 
\[
 (\rho_G)_M (m) (v \otimes u) = 0 
\]
for all $m \in M$ and all $u \in \theta C$. 

\end{rem*}

\subsection{The Clifford--Weil group}

We define the Clifford--Weil group and demonstrate that the $G$-full weight enumerator of a $G$-self-dual isotropic code is invariant under its action. 

\begin{Def}

Let $\Phi$ be an $R$-qmodule. We define the parabolic group $P (R, \Phi)$  associated with $(R, \Phi)$ as the semi-direct product $P (R, \Phi) = R^* \ltimes \Phi$, where the group operation is given by 
\[
 (r, \phi) (r', \phi') = (rr' , \phi [r'] + \phi') 
\]
for $r, r' \in R^*$ and $\phi, \phi' \in \Phi$. 

\end{Def}

\begin{Def}

Let $(M, \Phi)$ be a quadratic pair on $R$ and let 
$\rho_G =  (\theta (V^n), (\rho_G)_M, (\rho_G)_{\Phi})$ be a $G$-representation of $(M, \Phi)$. 
We define the action of $P (R, \Phi)$ on $\C [\theta (V^n)]$ by 
\[
 (r, \phi) e_u = \exp (2 \pi \sqrt{-1} (\rho_G)_{\Phi} (\phi) (u)) e_{ru}  
\]
for $r \in R^*$, $\phi \in \Phi$ and $u \in \theta (V^n)$. 
This action induces a group homomorphism $P (R, \Phi) \to \GL (\C [\theta (V^n)])$, where the image $(r, \phi) \mapsto m_r d_{\phi}$ is defined by 
\begin{align*}
 m_r (e_u) & = e_{ru}, \\ 
 d_{\phi} (e_u) & = \exp (2 \pi \sqrt{-1} (\rho_G)_{\Phi} (\phi) (u)) e_u, 
\end{align*}
for all $u \in \theta (V^n)$. We write the image of this homomorphism by 
\[
 P (\rho_G) \ceq \langle m_r, d_{\phi} \mid r \in R^*, \phi \in \Phi \rangle \subset \GL (\C [\theta (V^n)]). 
\]

\end{Def}

We observe that the $G$-full weight enumerator of $G$-isotropic code is invariant under the action of $P (\rho_G)$. 
Moreover, it can be shown that any vector in $\C [\theta (V^n)]$ invariant under $P (\rho_G)$ is a $\C$-linear combination of certain $G$-full weight enumerators. 

To prove this characterization, we employ the following lemma, which follows from a similar argument in \cite[Lemma 5.1.4]{NRS06}. 

\begin{lem} \label{Ru}

For any $v, w \in \theta (V^n)$, it holds that 
\[
 v \in R^* w 
\]
if and only if 
\[
 v \in R w 
\]

\end{lem}

\begin{thm} \label{parainv}

Let $G$ be a subgroup of the symmetric group of degree $n$, $(M, \Phi)$ a quadratic pair on $R$, and $\rho_G = (\theta (V^n), (\rho_G)_M, (\rho_G)_{\Phi})$ a $G$-representation of $(M, \Phi)$. 
Then, the vectors $\fwe_G (\theta C)$, where $C$ ranges over all $G$-isotropic codes (with respect to $\rho_G$), span the invariant subspace of $\C [\theta (V^n)]$ under the action of $P (\rho_G)$. That is, we have 
\[
 \C [\fwe_G (\theta C) \mid \text{$C$ is a $G$-isotropic code (with respect to $\rho_G$)}] = \C [\theta (V^n)]^{P (\rho_G)}. 
\]

\end{thm}

\begin{proof}

We first show that $\fwe_G (\theta C)$ is invariant under the action of $P (\rho_G)$ for any $G$-isotropic code $C$. Since $P(\rho_G)$ is generated by $m_r$ and $d_\phi$, it suffices to check the invariance under these generators for all $r \in R^*$ and all $\phi \in \Phi$. 
For $m_r$, since $\theta C$ is an $R$-module, we have 
\begin{align*}
 m_r (\fwe_G (\theta C)) 
& = m_r \left( \sum_{u \in \theta C} e_u \right) \\
& = \sum_{u \in \theta C} m_r (e_u) \\
& = \sum_{u \in \theta C} e_{ru} \\
& = \sum_{u \in \theta C} e_{u} \\
& = \fwe_G (\theta C). 
\end{align*}
For $d_{\phi}$, the $G$-isotropic property of $C$ implies 
$(\rho_G)_{\Phi} (\phi) (u) = 0$ for all $u \in \theta C$ . Thus, 
\begin{align*}
 d_{\phi} (\fwe_G (\theta C)) 
& = d_{\phi} \left( \sum_{u \in \theta C} e_u \right) \\
& = \sum_{u \in \theta C} d_{\phi} (e_u) \\
& = \sum_{u \in \theta C} \exp (2 \pi \sqrt{-1} (\rho_G)_{\Phi} (\phi) (u)) e_u \\
& = \sum_{u \in \theta C} \exp (0) e_u \\
& = \sum_{u \in \theta C} e_u \\
& = \fwe_G (\theta C). 
\end{align*}

Conversely, let $a = \sum_{u \in \theta (V^n)} a_u e_u 
\in \C [\theta (V^n)]^{P (\rho_G)}$. For all $r \in R^*$, we obtain 
\begin{align*}
 m_r ( a ) 
& = m_r \left( \sum_{u \in \theta (V^n)} a_u e_u \right) \\
& = \sum_{u \in \theta (V^n)} a_u m_r (e_u) \\
& = \sum_{u \in \theta (V^n)} a_u e_{ru} \\
& = a. 
\end{align*}
Hence this implies that $a_{r^{-1}u} = a_u$. In addition, for all $\phi \in \Phi$, 
\begin{align*}
 d_{\phi} (a) 
& = d_{\phi} \left( \sum_{u \in \theta (V^n)} a_u e_u \right) \\
& = \sum_{u \in \theta (V^n)} a_u d_{\phi} (e_u) \\
& = \sum_{u \in \theta (V^n)} a_u \exp (2 \pi \sqrt{-1} (\rho_G)_{\Phi} (\phi) (u)) e_u \\
& = a, 
\end{align*}
which means that $a_u = a_u \exp (2 \pi \sqrt{-1} (\rho_G)_{\Phi} (\phi) (u))$. 
Therefore, suppose $a_u \neq 0$, we have 
$\exp (2 \pi \sqrt{-1} (\rho_G)_{\Phi} (\phi) (u)) = 1$. Since $(\rho_G)_{\Phi} (\phi) (u) \in \QZ$, this yields $(\rho_G)_{\Phi} (\phi) (u) = 0$. 
Consequently, by applying Lemma \ref{Ru}, it follows that 
\[
 Ru = \{ru \in \theta (V^n) \mid r \in R \} 
\]
is a $G$-isotropic code and satisfies %the condition that 
\[
 \sum_{v \in Ru} a_v e_v = a_u \sum_{v \in Ru} e_v 
\]
for any $u \in \theta (V^n)$ with $a_u \neq 0$. 
Thus, we conclude that 
\begin{align*}
a & = \sum_{u \in \theta (V^n)} a_u e_u \\
& = \sum_{u \in \theta (V^n)} a_u \sum_{v \in Ru} e_v \\
& = \sum_{u \in \theta (V^n)} a_u \fwe_G (Ru). 
\end{align*}
\end{proof}

So, we have successfully generalized the results of \cite[Theorem 5.1.3, Remark 5.1.5]{NRS06}. Notably, our proof of Theorem \ref{parainv} offers a more direct approach than these theorems. 

Next, we define the transformations related to the duality of codes and introduce the Clifford--Weil group. 
Following \cite[Section 3]{NRS06}, We begin by defining the following algebraic structure. 

\begin{Def}

Let $(R, M, \psi)$ be a twisted ring. An element $\iota \in R$ is said to be a symmetric idempotent if it satisfies $\iota^2 = \iota$ and there exists an isomorphism of right $R$-modules: 
\[
 \iota R \cong \iota^J R, 
\]
where $r^J =  \psi^{-1} (\psi (r))$ for $r \in R$. 

\end{Def}

\begin{rem*}

If $\iota \in R$ is a symmetric idempotent, there exist elements 
$\mu_{\iota} \in \iota R \iota^J$ and $\nu_{\iota} \in \iota^J R \iota$ such that 
$\iota = \mu_{\iota} \nu_{\iota}$. 

\end{rem*}

To prove Theorem \ref{CWinv}, which asserts that the $G$-full weight enumerator of a $G$-self-dual isotropic code is invariant under the Clifford--Weil group, we employ the following lemma. This result is a $G$-analogue of \cite[Theorem 3.5.9]{NRS06}. 

\begin{lem} \label{iotaC}

Let $(R, M, \psi)$ be a twisted ring and let $\iota \in R$ be a symmetric idempotent. 
If a $G$-code $C$ is $G$-self-dual, then $\iota C$ is also $G$-self-dual. 

\end{lem}

\begin{Def}

Let $(R, M, \psi, \Phi)$ be a form ring and let 
$\rho_G = (\theta (V^n), (\rho_G)_M, (\rho_G)_{\Phi}, \beta)$ be a $G$-representation of $(R, M, \psi, \Phi)$.  
For a symmetric idempotent $\iota \in R$ with decomposition $\iota = \mu_{\iota} \nu_{\iota}$, we define the transformation $h_{\iota, \mu_{\iota}, \nu_{\iota}} \in \GL (\C [\theta (V^n)])$ by 
\begin{align*}
 h_{\iota, \mu_{\iota}, \nu_{\iota}} (e_u) =\abs{\iota (\theta (V^n))}^{- \frac{1}{2}} \sum_{v \in \iota (\theta (V^n))} \exp (2 \pi \sqrt{-1} \beta (v, \nu_{\iota} u)) e_{(1 - \iota) u + v}
\end{align*}
for each $u \in \theta (V^n)$. 

The Clifford--Weil group $\CW (\rho_G)$ is the subgroup of 
$\GL (\C [\theta (V^n)])$ generated by these transformations along with the parabolic generators: 
\[
 \CW (\rho_G) = \langle m_r, d_{\phi}, h_{\iota, \mu_{\iota}, \nu_{\iota}} 
\mid r \in R^*, \phi \in \Phi, \iota = \mu_{\iota} \nu_{\iota} \in R \text{ is a symmetric idempotent} \ \rangle. 
\]

\end{Def}

Finally, by extending the results of \cite[Theorem 5.5.1]{NRS06} to the $G$-invariant setting, we establish the following theorem. 

\begin{thm} \label{CWinv}

Let $G$ be a subgroup of the symmetric group of degree $n$, $(R, M, \psi, \Phi)$ a form ring, and $\rho_G = (\theta (V^n), (\rho_G)_M, (\rho_G)_{\Phi}, \beta)$ a $G$-representation of $(R, M, \psi, \Phi)$. If a $G$-code $C$ is both $G$-self-dual and $G$-isotropic (with respect to $\rho_G$), then its $G$-full weight enumerator $\fwe_G (\theta C)$ is invariant under the action of the Clifford--Weil group $\CW (\rho_G)$. 

\end{thm}

\begin{proof}

From Theorem \ref{parainv}, it suffices to show that $\fwe_G (\theta C)$ is invariant under $h_{\iota, \mu_{\iota}, \nu_{\iota}}$ for any symmetric idempotent $\iota \in R$. According to Lemma \ref{iotaC}, the problem reduces to the case $\iota = 1$ and 
$\mu_{\iota} = \nu_{\iota} = 1$. Hence, by Theorem \ref{fweMac}, we obtain 
\begin{align*}
 \fwe_G (\theta C)
& = \fwe_G ((\theta C)^{(\perp_G )_{\rho_G}}) \\
& = \frac{1}{\abs{\theta C}} \sum_{v \in \theta (V^n)} \sum_{u \in \theta C} \exp (2 \pi \sqrt{-1} \beta (v, u)) e_v \\
& = \abs{\theta (V^n)}^{- \frac{1}{2}} \sum_{v \in \theta (V^n)} 
\sum_{u \in \theta C} \exp (2 \pi \sqrt{-1} \beta (v, u)) e_v \\
& = \sum_{u \in \theta C} \abs{\theta (V^n)}^{- \frac{1}{2}} 
\sum_{v \in \theta (V^n)} \exp (2 \pi \sqrt{-1} \beta (v, u)) e_v \\
& = \sum_{u \in \theta C} h_{1, 1, 1} (e_u) \\
& = h_{1, 1, 1} (\sum_{u \in \theta C} e_u) \\
& = h_{1, 1, 1} (\fwe_G (\theta C)). 
\end{align*}
\end{proof}

So, we conclude that the $G$-full weight enumerator of a $G$-self-dual isotropic code is invariant under the Clifford--Weil group. 

Furthermore, based on Theorem \ref{parainv}, we are led to the following fundamental question regarding the span of weight enumerators: 

\vspace{7mm}

\noindent \textbf{Future Work.}

Let $\rho_G = (\theta (V^n), (\rho_G)_M, (\rho_G)_{\Phi}, \beta)$ be a $G$-representation of a form ring $(R, M, \psi, \Phi)$. 
Do the vectors $\fwe_G (\theta C)$, where $C$ ranges over all $G$-self-dual isotropic codes (with respect to $\rho_G$), span the invariant subspace of $\C [\theta (V^n)]$ under the action of $\CW (\rho_G)$? 

\begin{que}
Does the following equality hold: 
\[
 \C [\fwe_G (\theta C) \mid \text{$C$ is a $G$-self-dual isotropic code (with respect to $\rho_G$)}] = \C [\theta (V^n)]^{\CW (\rho_G)} ? 
\]
\end{que}

This question is a natural extension of the \textit{The Weight Enumerator Conjecture} (Conjecture 5.5.2. (1) in \cite{NRS06}) to the setting of $G$-representations. From Theorem \ref{CWinv}, we have already established the forward inclusion: 
\[
 \C [\fwe_G (\theta C) \mid \text{$C$ is a $G$-self-dual isotropic code (with respect to $\rho_G$)}] 
\subset \C [\theta (V^n)]^{\CW (\rho_G)}
\]
So, the essence of the problem lies in determining whether the reverse inclusion holds: 
\[
 \C [\fwe_G (\theta C) \mid \text{$C$ is a $G$-self-dual isotropic code (with respect to $\rho_G$)}] \supset \C [\theta (V^n)]^{\CW (\rho_G)}. 
\]

\end{document}